%% file: main.tex
\documentclass[11pt,letterpaper,reqno]{amsart}

\usepackage[margin=1.02in]{geometry}
\usepackage[T1]{fontenc}
\usepackage[utf8]{inputenc}
\usepackage{lmodern}
\usepackage{microtype,comment}
\usepackage{tikz-cd}

\usepackage{amssymb,mathtools,mathrsfs}
\usepackage{graphicx,booktabs,enumitem}

\usepackage{xcolor}
\definecolor{paperblue}{HTML}{124878}
\usepackage[
  colorlinks=true,
  linkcolor=paperblue,
  citecolor=paperblue,
  urlcolor=paperblue,
  linktoc=all,
  bookmarksnumbered=true,
  pdfencoding=auto,
  psdextra
]{hyperref}
\usepackage[
  backend=biber,
  style=alphabetic,
  maxnames=10,
  giveninits=true,
  sortcites=true,      
  sorting=nyt          
]{biblatex}
\AtEveryBibitem{%
  \ifboolexpr{
    test {\iffieldundef{year}}
    and
    test {\iffieldundef{issue}}
  }
    {\renewbibmacro*{issue+date}{}}
    {}%
}

\usepackage{aliascnt}
\usepackage[capitalize,nameinlink,noabbrev]{cleveref}
\input{macros.tex}
\numberwithin{equation}{section}

\title[Equidistribution of periodic points for automorphisms]{Equidistribution of saddle periodic points for automorphisms on compact K\"ahler manifolds}
\author{Muhan Luo}
\address{Department of Mathematics, National University of Singapore, Singapore 119076}
\email{luomuh98@gmail.com}

\author{Qi Zhou}
\address{Department of Mathematics, National University of Singapore, Singapore 119076}
\email{e1124859@u.nus.edu}

\date{} 
\subjclass[2020]{}

\begin{document}

\begin{abstract}
    Let $f$ be a holomorphic automorphism of a compact K\"ahler manifold $X$ with simple cohomological action. We prove that the saddle periodic points of $f$ equidistribute with respect to the equilibrium measure. In particular, our result applies to holomorphic automorphisms of compact K\"ahler surfaces with positive topological entropy.
\end{abstract}

\maketitle
\tableofcontents
%\pagebreak 

\section{Introduction and main result}

The distribution of periodic points is a central question in complex dynamics. A natural problem is to determine whether their normalized counting measures converge to an invariant probability measure that describes the statistical behaviour of the system. This phenomenon has been established in a variety of complex dynamical systems through the work of Lyubich \cite{Lyubich83FAA} on rational maps on $\Pb^1$, Bedford--Lyubich--Smillie \cite{BLS93b} on polynomial automorphisms of $\C^2$, Briend--Duval \cite{Briend-Duval99} on endomorphisms of $\Pb^k$, Cantat \cite{Cantat01Acta} on automorphisms of projective K3 surfaces, Dinh--Sibony \cite{DS03JMPA} on polynomial-like maps, Diller--Dujardin--Guedj \cite{DDG10ENS} on rational maps of small topological degree on projective surfaces, and Dinh--Sibony \cite{DS16} on H\'enon-type automorphisms of $\C^k$. See also Dujardin \cite{DurjaDuke} for results on birational maps of rational surfaces. The authors' recent work \cite{LZ25} also proves this result for H\'enon-like maps of arbitrary dimension. These are only a few representative contributions to a much broader literature, and we make no attempt to give an exhaustive account here. 

However, for holomorphic automorphisms on general compact K\"ahler manifolds with simple action on cohomology, the problem has been open for decades, even in the case of dimension two \cite{Cantat01Acta}. See \cite{DS17} for more discussions on this topic. In this paper, we completely solve this question. 
\medskip

Let $X$ be a compact K\"ahler manifold of complex dimension $k\geq 1$, and let $f\colon X\to X$ be a holomorphic automorphism. For $0\leq q\leq k$, the $q$-th \textit{dynamical degree} of $f$ is defined to be the spectral radius of $f^*$ acting on $H^{q,q}(X,\C)$. Note that by a result of Gromov-Yomdin \cite{Gromov2003,Yomdin1987}, the topological entropy of $f$ equals $h_{\mathrm{top}}(f)=\max_{0\leq q\leq k}\log d_q$.

\begin{definition}\label{def:simple}
The cohomological action of $f$ is \emph{simple} if there is an index $p$ such that
    \begin{equation}\label{eq:gap}
        d:=d_p(f)>\max_{q\ne p}d_q(f),
    \end{equation}
and $d$ is an algebraically simple eigenvalue of $f^*|_{H^{p,p}(X,\C)}$, with every other eigenvalue of this operator having modulus strictly smaller than $d$. We call $d$ the \textit{main dynamical degree} of $f$.
\end{definition}

Since $d_0(f)=d_k(f)=1$, these assumptions imply $1\leq p\leq k-1$ and $d>1$. The dynamics of holomorphic automorphisms with simple cohomological action has been studied intensively, see for example \cite{DS10JAG,dTH-Dinh12AM}. See also \cite{OT15JMST} for an explicit example of automorphisms with positive entropy in higher dimension. 
 
Here we recall some of the results to state our main theorem. For more details, see \cref{sec:auto}. There are two positive closed currents $T^+$ and $T^-$, of bi-degrees $(p,p)$ and $(k-p,k-p)$, respectively, such that
\begin{equation}\label{eq:green}
    f^*T^+=dT^+, \qquad f_*T^-=dT^-.
\end{equation}
They are called the \textit{Green currents}. The \textit{equilibrium measure} of the system $f$ is defined by
\[
    \mu:=T^+\wedge T^-.
\]
It is a probability measure after normalizing $T^\pm$. It is also a measure of maximal entropy with $h_\mu(f)=h_{\mathrm{top}}(f)=\log d$. Throughout the paper, a \textit{periodic point of period $n$} means a fixed point of $f^n$; its least period is allowed to divide $n$. Our main result is the following.

\begin{theorem}\label{thm:main}
Assume that the cohomological action of $f$ is simple in the sense of \cref{def:simple}. Let $P_n$ be the set of isolated periodic points of period $n$ of $f$, counted with multiplicities, and let $SP_n\subset P_n$ be the saddle periodic points. Let $Q_n$ be either $P_n$ or $SP_n$. Then we have
    \[
        d^{-n}\sum_{z\in Q_n}\delta_z\to \mu
    \]
as $n$ goes to infinity, where $\delta_z$ denotes the Dirac measure at $z$.
\end{theorem}

For automorphisms of compact K\"ahler surfaces, having positive entropy implies simple action on cohomology. Our main theorem therefore yields the following corollary. This applies, in particular, to automorphisms of K3 surfaces with positive entropy, whose dynamics have been extensively studied; see, for example, \cite{Cantat01Acta,FT23CJM,FT24JGA,FT21AJM}. The corollary settles the equidistribution problem for general compact Kähler surfaces, which has remained open since Cantat's work \cite{Cantat01Acta} on projective K3 surfaces in 2001.

\begin{corollary}
    Let $f$ be a holomorphic automorphism of a compact K\"ahler surface $X$ with positive topological entropy $h:=h_{\mathrm{top}}(f)$, and let $\mu$ be its equilibrium probability measure. For each $n\geq1$, let $P_n$ be the set of isolated fixed points of $f^n$ with multiplicities, and let $SP_n\subset P_n$ be the subset of saddle points. Let $Q_n$ be either $P_n$ or $SP_n$. Then we have
\[
    e^{-nh}\sum_{z\in Q_n}\delta_z\;\longrightarrow\;\mu
\]
as $n\to\infty$. 
\end{corollary}

\begin{remark}
We can further strengthen the result of \cref{thm:main}. Let $\chi^{+}>0$ and $\chi^{-}<0$ be the positive and negative Lyapunov exponents of $\mu$ with the smallest absolute value, respectively. For $0<\vep < \min\{ \chi^+, -\chi^- \} $, let $SP_{n ,\vep}\subset SP_n$ consist of the points $x$ for which $D_x f^n$ has $p$ eigenvalues whose absolute values are larger than $\exp \big(n(\chi^+-\vep)\big)$ and $(k-p)$ eigenvalues whose absolute values less than $\exp\big(n(\chi^-+\vep)\big)$. Every such point in $SP_{n ,\vep}$ is a saddle $n$-periodic point with $p$ repelling and $(k-p)$ attracting complex directions. Then, we have 
\[
    d^{-n}\sum_{z\in SP_{n ,\vep}}\delta_z \longrightarrow\mu.
\]

In fact, since $\mu$ is mixing, the lower bound for the number of points in $SP_{n,\vep}$ in \cite[Theorem~9.1]{ALP24MAMS}, together with the upper bound \eqref{eq:DNV}, implies that $\#(P_n\setminus SP_{n,\vep})=o(d^n)$. The conclusion therefore follows from \cref{thm:main}. To keep the presentation concise, we omit the verification of the technical hypotheses of \cite[Theorem~9.1]{ALP24MAMS} and content ourselves with the current formulation of \cref{thm:main}.
\end{remark}

Our argument also gives another proof of the uniqueness of the measure of maximal entropy in our setting, which is first proven by de Th\'elin-Dinh \cite[Theorem 1.2]{dTH-Dinh12AM}. See also \cite[Theorem 0.1]{Cantat01Acta}. 

\begin{corollary}
    Let $f$ be as in \cref{thm:main}. Then the equilibrium measure $\mu$ is the unique $f$-invariant probability measure of maximal entropy. 
\end{corollary}

Indeed, if $\nu$ is another invariant ergodic probability measure of maximal entropy, de Th\'elin \cite[Corollary 2]{deT08Invent} proves that $\nu$ is hyperbolic. Then, one can apply \cref{prop:symbolic} and \cref{lem:coding} to show that saddle periodic points equidistribute to the measure $\nu$. Hence $\nu=\mu$.
\medskip

Our proof of \cref{thm:main} uses the symbolic coding theorem of Ben Ovadia \cite{Ovadia18JMD} and a theorem of Buzzi \cite{Buzzi20JMD}. These results yield an injective H\"older-continuous map $\pi:(\Sigma,\sigma)\to (X,f)$ with $f\circ \pi = \pi\circ \sigma$, where $(\Sigma,\sigma)$ is an irreducible countable Markov shift, and periodic points in $\Sigma$ are mapped to the saddle periodic points on $X$. Relevant definitions on Markov shifts can be found in \cref{sec:CMS}. The existence of a measure of maximal entropy ensures that the shift is positive recurrent, and one can give an explicit characterization of the measure. On the other hand, by Dinh-Nguyen-Vu \cite[Theorem~1.3]{DNV}, we have the estimate
\[
    \#P_n\leq d^n+o(d^n).
\]
By injectivity of $\pi$, this transfers to a bound on the number of periodic points of $(\Sigma,\sigma)$. We show that this bound yields equidistribution on the system $(\Sigma,\sigma)$ and finally gives the result of \cref{thm:main}.
\medskip

The paper is organized as follows. \cref{sec:auto} recalls the required facts about Green currents, equilibrium measures, and periodic point counts. \cref{sec:CMS} reviews basic notions and properties of countable Markov shifts. \cref{sec:symbolic-proof} establishes the equidistribution result for a symbolic dynamics model. \cref{sec:main} completes the proof of \cref{thm:main}.

\medskip
\noindent\textbf{Acknowledgement.} The first author is supported by the NUS grant A-8002488-00-00. The second author is supported by the President's Graduate Fellowship. The main idea of the proof is suggested by GPT-6 Astra. The authors verified the idea, filled in the details, and wrote the paper. The authors are fully responsible for all mathematical content. We also would like to mention that before the recent surge in AI use, the authors had obtained the same result for projective manifolds following another approach. 

\section{Holomorphic automorphisms on compact K\"ahler manifolds}\label{sec:auto}

In this section, we recall basic properties of holomorphic automorphisms on compact K\"ahler manifolds. All result mentioned here along with their proofs can be found in \cite{DS05JAMS,DS10JAG,DNV}. 

Let $f$ be a holomorphic automorphism of a compact K\"ahler manifold $X$ of complex dimension $k$. Suppose the action of $f^*$ on cohomology is simple and the main dynamical degree is $d:=d_p(f)>1$ for some $1\leq p\leq k-1$. By Dinh--Sibony \cite[Theorem~4.2.1]{DS10JAG}, there exist nonzero positive closed Green currents $T^+$ and $T^-$ of bidegrees $(p,p)$ and $(k-p,k-p)$, respectively, satisfying
\[
    f^*T^+=dT^+,\qquad f_*T^-=dT^-.
\]
They have H\"older continuous super-potentials, so their wedge product is well-defined. The equilibrium measure is then defined to be
\[
    \mu:=T^+\wedge T^-.
\]
After normalization, $\mu$ is a probability measure. By \cite[Theorem~4.4.2]{DS10JAG}, $\mu$ is invariant and mixing, hence ergodic, and
\begin{equation}\label{eq:entropy}
    h_\mu(f)=h_{\mathrm{top}}(f)=\log d.
\end{equation}

Moreover, $\mu$ is hyperbolic. Recall that an invariant measure is \textit{hyperbolic} if its Lyapounov exponents are non-zero. In fact, we have a more precise estimate of the Lyapounov exponents: $\mu$ admits $p$ positive Lyapounov exponents larger than or equal to $\frac12\log\frac{d_p}{d_{p-1}}$ and
$k-p$ negative exponents at most equal to $-\frac12\log\frac{d_p}{d_{p+1}}$.

As in the introduction, let $P_n$ denote the isolated fixed points of $f^n$, counted with multiplicities. Under our cohomological assumptions, Dinh--Nguyen--Vu \cite[Theorem~1.3]{DNV} give
\begin{equation}\label{eq:DNV}
    \#P_n\leq d^n+o(d^n)\qquad\text{as }n\to\infty.
\end{equation}
This is crucial for the derivation of our main result. One should note that the multiplicity of saddle periodic points is always 1.

\section{Countable Markov shifts}\label{sec:CMS}

\subsection{Basic notions}

We recall basic results of countable Markov shifts. All results mentioned here can be found in \cite[Chapter 7]{Kitchens} and \cite[Chapter 1]{Durrett99ESP}. Let $\cG=(\cV, E)$ be a directed graph with a countable collection of vertices $\cV$ s.t. every vertex has at least one ingoing and one outgoing edge in $E$. We always assume that there is at most one directed edge from \(u\) to \(v\) for every ordered pair of vertices \((u,v)\). We write $u\to v$ when such an edge exists. The \textit{adjacency matrix} is defined to be the countable matrix $A$ such that the vertices of the graph are the indices for the rows and columns of $A$, and for any two vertices $u,v$, 
\[
    A_{uv}=
    \begin{cases}
        1, & u\to v;\\
        0, & \text{otherwise}.
    \end{cases}
\]
The \textit{countable Markov shift} associated to $\cG$ is the set
 \[
    \Sigma=\{(v_i)_{i\in\Z}\in\cV^\Z: A_{v_iv_{i+1}}=1,\ \forall i\in\Z\}.
\]
equipped with the left-shift
\[
    \sigma \colon \Sigma \to \Sigma, \qquad \sigma\bigl((v_i)_{i\in\mathbb{Z}}\bigr)= (v_{i+1})_{i\in\mathbb{Z}},
\]
and the metric
\[
    d(u,v):= \exp\left(-\min\left\{n\in\mathbb{N}_0 :u_n\neq v_n \text{ or } u_{-n}\neq v_{-n}\right\}\right)
\]
for \(u=(u_n)_{n\in\mathbb{Z}}\), \(v=(v_n)_{n\in\mathbb{Z}}\). With this metric \(\Sigma\) is a complete separable metric space, i.e., a metric Polish space.

A \emph{Markovian subshift} of \(\Sigma(\mathcal{G})\) is a subset of the form \(\Sigma(\mathcal{G}')\), where \(\mathcal{G}'=(\mathcal{V}', E')\) is a subgraph of \(\mathcal{G}\) satisfying the same condition for $\mathcal{G}$ . The following definition is due to Sarig: given a countable Markov shift \(\Sigma\), 
\begin{equation}\label{eq:regular}
    \Sigma^\#:=\left\{(v_i)_{i\in\mathbb{Z}}\in\Sigma : \exists\,v',w'\in\mathcal{V}(\Sigma),\ \exists\,n_k,m_k\uparrow\infty\ \text{s.t.}\  v_{n_k}=v', \  v_{-m_k}=w',\  \forall\,k\geq 1 \right\}.
\end{equation}
By the Poincaré recurrence theorem, every \(\sigma\)-invariant probability measure $\mu$ is carried by \(\Sigma^\#\), i.e., $\mu(\Sigma^{\#})=1$. Furthermore, every periodic point of \(\sigma\) is in \(\Sigma^\#\). Note that \(\Sigma^\#\) is not necessarily a closed subset of \(\Sigma\).

A \textit{path} of length $n$ is a sequence of vertices $(u_0,u_1,...,u_n)$ such that $u_i\to u_{i+1}$ for all $0\leq i\leq n-1$. This path is called a \textit{loop} if $u_0 = u_n$. We can consider the product of $A$ with itself by the usual law of matrix product. In particular, $(A^n)_{ab}$ is the number of paths $(u_0,u_1,...,u_n)$ such that $u_0=a$ and $u_n=b$. By convention, $(A^n)_{ab}$ takes $\infty$ if there are infinitely many paths of length $n$ from $a$ to $b$.

\begin{definition}
    The Markov shift $\Sigma$ associated with $\cG$ is \textit{irreducible} if for all $u,v\in\cV$, there exists a path from $u$ to $v$; equivalently, $\forall u,v\in\cV$, $\exists n>1$ such that $(A^n)_{uv}>0$. The \textit{period} of an irreducible Markov shift $\Sigma$ is the greatest common divisor of the lengths of loops in the graph. Equivalently, for any fixed $u\in\cV$, the \textit{period} is defined by
    \[
        q:=\gcd\{n\geq1:(A^n)_{uu}>0\},
    \]
    which is independent of the choice of $u$. An irreducible Markov shift $\Sigma$ is \textit{aperiodic} if its period is 1.
\end{definition}

\subsection{Positive recurrence and entropy}

\begin{definition}
    When $\Sigma$ is an irreducible countable Markov shift, the \textit{Perron value} of $\Sigma$ is defined to be
\[
    \lambda:={\limsup}_{n\to\infty}((A^n)_{uu})^{1/n}.
\]
It is independent of the choice of $u$ and takes values in $[1, \infty]$.
\end{definition}

In the following, we assume that the countable Markov shift $\Sigma$ is irreducible and its Perron value $\lambda$ is finite. When $\lambda < \infty$, it is easy to see that $(A^n)_{uv}$ is always finite for any $u,v$ and $n\geq 1$. We now discuss recurrence properties of an irreducible countable Markov shift $\Sigma$. For any pair of vertices $u$ and $v$, define 
\[
    a_{uv}(0)=\delta_{uv}\quad \text{and}\quad a_{uv}(n)=(A^n)_{uv}.
\]
Let
\[
    l_{uv}(0)=0, \quad l_{uv}(1)=A_{uv}, \quad \text{and}\quad l_{uv}(n+1)=\sum_{w\neq u} l_{uw}(n)A_{wv}.
\]
The coefficient $l_{uv}(n)$ is the number of the paths that go from $u$ to $v$ in $n$ steps without returning to $u$ at any time prior to $n$. 

\begin{definition}
     An irreducible countable Markov shift $\Sigma$ with finite Perron value $\lambda$ is \textit{recurrent} if for some $u\in\cV$,
     $$
     \sum_{n=1}^{\infty} a_{uu}(n)/\lambda^n= \infty.
     $$
     It is \textit{positive recurrent} if it is recurrent and
     \[
       \sum_{n=1}^\infty nl_{uu}(n)/\lambda^n<\infty.
     \]
     Both definitions are independent of the choice of $u$.
\end{definition}

The following is the Perron-Frobenius theorem for countable Markov shifts. For reference, see \cite[Theorem 7.1.3]{Kitchens}.
\begin{lemma}\label{Perron-Frobenius}
    Let $\Sigma$ be an irreducible countable Markov shift with finite Perron value $\lambda$. If $\Sigma$ is recurrent, $\lambda$ has strictly positive left and right eigenvectors $r=(r_u)$ and $\ell=(\ell_u)$ with $rA=\lambda r$ and $A\ell=\lambda \ell$. And the eigenvectors are unique up to constant multiples. Moreover, if $\Sigma$ is positive recurrent, we have that $r\cdot \ell=\sum_u r_u\cdot l_u<\infty$.
\end{lemma}

After normalization, we can assume $r\cdot \ell=1$.

The \textit{Gurevich entropy} of the Markov shift $\Sigma$ associated with $\cG$ is given by
\[
     h(\Sigma):=\sup\{h_{top}(\Sigma(\cG'),\sigma):\cG' \text{ is a finite subgraph}\}.
\]
The variational principle is still valid for Markov shifts:
\begin{theorem}[\cite{Gurevich69}]\label{variational}
    For an irreducible countable Markov shift $\Sigma$ with Perron value $\lambda$ , we have 
    \[
         h(\Sigma)=\sup_{\mu} h_\mu(\sigma)=\log \lambda
    \]
    where $\mu$ runs over all invariant Borel probability measures. One should understand that if one value in the equality is infinite, the other two values are also infinity.
\end{theorem}

The next theorem gives a necessary and sufficient condition for the existence of a measure of maximal entropy when the Markov shift is irreducible.

\begin{theorem}[\cite{Gurevich70}]\label{thm:MME-posi recurr}
     Let $\Sigma$ be an irreducible countable Markov shift with Perron value $\lambda< \infty$. Then, it admits a measure of maximal entropy if and only if $\Sigma$ is positive recurrent and such a measure is unique if it exists.
\end{theorem}

From now on, assume that $(\Sigma, \sigma)$ is an irreducible countable Markov shift with $0<h(\Sigma)<\infty$ and admits a measure of maximal entropy $\mu$. We can give a detailed description about the measure $\mu$. By \cref{thm:MME-posi recurr}, $\Sigma$ is positive recurrent. Let $r$ and $\ell$ be the eigenvectors in \cref{Perron-Frobenius}, normalized by $r\cdot\ell=1$. Define
\begin{equation}\label{eq:transition-stationary}
    K_{uv}:=\frac{A_{uv}\ell_v}{\lambda\ell_u}, \qquad \rho_u:=r_u\ell_u.
\end{equation}
The positivity of $\ell_u$ makes $K_{uv}$ well-defined. Furthermore,
\[
    \sum_{v\in\cV}K_{uv}=\frac{(A\ell)_u}{\lambda\ell_u}=1, \qquad\sum_{u\in\cV}\rho_u=r\cdot\ell=1,
\]
and
\[
    (\rho K)_v=\sum_{u\in\cV} r_u\ell_u\frac{A_{uv}\ell_v}{\lambda\ell_u}=\frac{\ell_v}{\lambda}(rA)_v=r_v\ell_v=\rho_v.
\]
Thus $K$ is a stochastic matrix and $\rho$ is a stationary probability distribution of $K$. For integers $m,n$ with $n\geq 0$, and an admissible path $u_0\to u_1\to \cdots \to u_n $, define
\[
    [u_0,\ldots,u_n]_m:=\{x\in\Sigma:x_j=u_{j-m}\text{ for }m\leq j \leq n+m \}.
\]
Prescribe cylinder probabilities by
\begin{equation}\label{eq:MME}
    \begin{aligned}
        \mu([u_0,\ldots,u_n]_m)
        &=\rho_{u_0}\prod_{j=m}^{n+m-1}K_{u_{j-m}u_{j-m+1}}\\
        &=\lambda^{-n}r_{u_0}\ell_{u_n},
    \end{aligned}
\end{equation}
where an empty product equals one. One can see that the measure of an admissible cylinder depends only on the number of the edges and the starting and ending vertices. In this way, we can get a unique probability measure on $\Sigma$ by Kolmogorov extension theorem.

We also briefly introduce the Markov chain associated with a countable Markov shift $\Sigma$ or its stochastic matrix $K$. Define the coordinate random variables
\[
    X_n:\Sigma\to\cV,\qquad X_n((v_i)_{i\in\Z})=v_n,\qquad n\in\Z.
\]
Under $\mu$, the process $(X_n)_{n\in\Z}$ is a stationary time-homogeneous Markov chain with transition matrix $K$ and one-coordinate distribution $\rho$. More precisely the cylinder formula gives
\[
    \mu(X_{n+m+1}=v\mid X_m=u_0,\ldots,X_{m+n}=u_n)=K_{u_nv}
\]
for every integers $n,m$ with $n\geq 0$. Also, $\mu(X_n=u)=\rho_u$ for every $n\in\Z$ and $u\in\cV$.

One can also consider multistep transition probabilities for the Markov chain constructed above.
By \cite[Theorem~1.1]{Durrett99ESP}, for $m\in\Z$ and $n\geq0$,
\[
    \mu(X_{n+m}=v\mid X_m=u)=(K^n)_{uv}.
\]

Using \eqref{eq:transition-stationary}, we obtain
\begin{equation}\label{eq:matrix-identity}
    (K^n)_{uv}=\lambda^{-n}\frac{\ell_v}{\ell_u}(A^n)_{uv}, \qquad n\geq0.
\end{equation}
The identity is immediate for $n=0$, and the induction step follows from
\[
\begin{aligned}
    (K^{n+1})_{uv}
    &=\sum_{w\in\cV}(K^n)_{uw}K_{wv}\\
    &=\sum_{w\in\cV}
       \left(\lambda^{-n}\frac{\ell_w}{\ell_u}(A^n)_{uw}\right)
       \left(\frac{A_{wv}\ell_v}{\lambda\ell_w}\right)\\
    &=\lambda^{-(n+1)}\frac{\ell_v}{\ell_u}
       \sum_{w\in\cV}(A^n)_{uw}A_{wv}\\
    &=\lambda^{-(n+1)}\frac{\ell_v}{\ell_u}(A^{n+1})_{uv}.
\end{aligned}
\]

We will need the following convergence theorem for Markov chains. It states that, starting from any fixed state, the distribution at time $n$ converges to the stationary distribution.

\begin{theorem}[\cite{Durrett99ESP}]\label{markov-aperiodic}
Let $\Sigma$ be a positive recurrent aperiodic irreducible countable Markov shift with finite Perron value. Let $K$ and $\rho$ be its associated stochastic matrix and stationary probability distribution, respectively. Then, for every $u,v\in\mathcal V$,
\[
    \lim_{n\to\infty}(K^n)_{uv}=\rho_v.
\]
\end{theorem}

As a corollary, we also have the following version without the aperiodicity assumption for diagonal entries of the transition matrices.

\begin{corollary}\label{lem:markov-convergence}
Let $\Sigma$ be a positive recurrent irreducible countable Markov shift with finite Perron value. Suppose $q$ is its period. Let $K$ and $\rho$ be its associated stochastic matrix and stationary probability distribution, respectively. Then, for every $u\in\mathcal V$,
\[
    \lim_{n\to\infty}(K^{qn})_{uu}=q\rho_u.
\]
\end{corollary}

Finally, we prove that every ergodic shift-invariant probability measure on $\Sigma$ is supported on an irreducible countable Markov subshift, as shown by the following lemma. This is already known in the works \cite{Sarig13JAMS, Katok07JMD,Buzzi09ETDS,Ovadia18JMD}. We include a proof here for completeness.

\begin{lemma}\label{lem:ergodic-component}
Let $(\Sigma,\sigma)$ be a countable Markov shift with alphabet $\mathcal V$, and let $\nu$ be an ergodic $\sigma$-invariant Borel probability measure. Then there is a strongly connected component $\mathcal C$ of the underlying graph such that
\[
    \supp(\nu)\subset\Sigma({\mathcal C}), \qquad \Sigma(\mathcal{C}) \text{ is the irreducible countable Markov subshift associated with } \mathcal{C}.
\]
Moreover, every symbol $a$ with $\nu([a]_0)>0$ belongs to $\mathcal C$.
\end{lemma}

\begin{proof}
    Write $ \cV_+:=\{a\in\mathcal V:\nu([a]_0)>0\}$. Since $\mathcal V$ is countable and the cylinders $[a]_0$ partition $\Sigma$, the set $\mathcal V_+$ is nonempty. Fix $a,b\in\mathcal V_+$ and consider
\[
    V_b:=\bigcup_{l\geq1}[b]_l.
\]
This is the set of sequences in which the symbol $b$ occurs at some strictly positive time. We have $\sigma^{-1}V_b\subset V_b$, and invariance of $\nu$ gives
\[
    \nu(\sigma^{-1}V_b)=\nu(V_b).
\]
Together with the inclusion above, this implies $\nu\bigl(V_b\mathbin{\triangle}\sigma^{-1}V_b\bigr)=0$. Thus $V_b$ is invariant modulo a $\nu$-null set. By ergodicity, $\nu(V_b)$ is either $0$ or $1$. Since
\[
    \nu(V_b)\geq\nu([b]_1)=\nu(\sigma^{-1}[b]_0)=\nu([b]_0)>0,
\]
we obtain $\nu(V_b)=1$ and hence $\nu([a]_0\cap V_b)=\nu([a]_0)>0$.

On the other hand,
\[
    [a]_0\cap V_b=\bigcup_{l\geq1}\bigl([a]_0\cap[b]_l\bigr).
\]
Consequently, there is an integer $l\geq1$ such that $\nu([a]_0\cap[b]_l)>0$. Choose a sequence $u$ in this intersection. Then $u_0=a$ and $u_l=b$, and admissibility gives a directed path
\[
    a=u_0\to u_1\to\cdots\to u_l=b.
\]
Interchanging $a$ and $b$ gives a directed path from $b$ to $a$. It follows that all symbols in $\mathcal V_+$ lie in one irreducible component, denoted by $\mathcal C$.

We next show that the corresponding Markov subshift has full $\nu$-measure. For every $a\notin\mathcal V_+$ and every $j\in\mathbb Z$, shift invariance gives
\[
    \nu([a]_j)=\nu(\sigma^{-j}[a]_0)=\nu([a]_0)=0.
\]
Both $\mathcal V$ and $\mathbb Z$ are countable. Therefore
\[
    N:=\bigcup_{j\in\mathbb Z}\ \bigcup_{a\in\mathcal V\setminus\mathcal V_+}[a]_j
\]
is a $\nu$-null set. Every sequence outside $N$ uses only symbols in $\mathcal V_+$, and hence only symbols in $\mathcal C$. Because $\mathcal{C}$ is an irreducible component, we conclude that sequences outside $N$ must belong to $\Sigma(\mathcal{C})$. Thus we have $\nu(\Sigma_{\mathcal C})=1$. As $\Sigma(\mathcal{C})$ is closed in $\Sigma$, we conclude that $\supp(\nu)\subset \Sigma(\mathcal{C})$.
\end{proof}

\section{Equidistribution in the symbolic model}\label{sec:symbolic-proof}
We establish the equidistribution of periodic points for an irreducible countable Markov shift admitting a measure of maximal entropy, under a specific upper bound on the number of periodic points. This result provides the symbolic model for the proof of our main theorem.

\begin{proposition}\label{prop:symbolic}
Let $(\Sigma,\sigma)$ be an irreducible countable Markov shift. Let $A$ be its adjacency matrix and $\lambda$ be its Perron value. Suppose the Gurevich entropy $h(\Sigma)=\log \lambda$ is finite and strictly positive and $\Sigma$ admits a maximal-entropy probability measure $\nu$. Assume that $\Fix(\sigma^n)$ is finite for every $n$ and
\begin{equation}\label{eq:trace-upper}
    \limsup_{n\to\infty}\lambda^{-n}\#\Fix(\sigma^n)\leq1.
\end{equation}
Then
\[
    \eta_n:=\lambda^{-n}\sum_{u\in\Fix(\sigma^n)}\delta_u\to \nu
\]
as $n\to\infty$.
\end{proposition}

\begin{proof}

By \cref{thm:MME-posi recurr}, the shift is positive recurrent. For $n\geq1$, the diagonal entry $(A^n)_{uu}$ counts the points $v\in\Fix(\sigma^n)$ with $v_0=u$. Let $q$ be the period of $\Sigma$. Thus, for any finite set $F\subset\cV$, \cref{lem:markov-convergence}, \eqref{eq:matrix-identity}, and \eqref{eq:trace-upper} imply
\begin{align*}
 q\sum_{u\in F}\rho_u
    &=\lim_{n\to\infty}\lambda^{-qn}
        \sum_{u\in F}(A^{qn})_{uu}\\
    &\leq\limsup_{n\to\infty}\lambda^{-qn}
        \#\Fix(\sigma^{qn})\leq1.
\end{align*}
Exhausting $\cV$ by finite sets gives $q\leq1$, hence $q=1$. Therefore, $\Sigma$ is aperiodic. Then, \cref{markov-aperiodic} and \eqref{eq:matrix-identity} yields
\begin{equation}\label{eq:matrix-limit}
    \lambda^{-n}(A^n)_{uv}\longrightarrow r_{v}\ell_{u}.
\end{equation} Fix an admissible cylinder $C=[u_0,\cdots, u_m]_0$. For $n>m$, completing its prescribed path to an $n$-periodic sequence amounts to choosing a path of length $n-m$ from $u_m$ to $u_0$. Consequently,
\[
    \#\bigl(\Fix(\sigma^n)\cap C\bigr)=(A^{n-m})_{u_m,u_0}.
\]
Using \eqref{eq:matrix-limit} and \eqref{eq:MME}, we obtain
\begin{align*}
 \eta_n(C)
    &=\lambda^{-m}\lambda^{-(n-m)}(A^{n-m})_{u_m,u_0}\\
    &\longrightarrow\lambda^{-m}r_{u_0}\ell_{u_m}
    =\nu(C).
\end{align*}
Since both $\eta_n$ and $\nu$ are shift-invariant, we can obtain the convergence for cylinders on any finite interval. Note that the intersection of finite cylinders is empty or a disjoint union of several cylinders, the above convergence result also holds for finite union of cylinders because of inclusion-exclusion principle. Because the cylinders on finite intervals form a countable topology basis for $\Sigma$, for any open set $G\subset\Sigma$, we can write $G=\bigcup_{j\geq1}C_j$, where each $C_j$ is a cylinder contained in $G$. For every $N\geq1$,
\[
    \liminf_{n\to\infty}\eta_n(G)\geq\lim_{n\to\infty}\eta_n\Bigl(\bigcup_{j=1}^{N}C_j\Bigr)=\nu\Bigl(\bigcup_{j=1}^{N}C_j\Bigr).
\]
Letting $N\to\infty$ yields $\liminf_{n\to\infty}\eta_n(G)\geq\nu(G)$ for any open subset $G$. This implies $\eta_n$ converges to $\nu$. Therefore, we finish the proof.
\end{proof}

\section{Proof of the main theorem}\label{sec:main}

\begin{proposition}\label{lem:coding}
Let $g$ be a $C^{1+\alpha}$ diffeomorphism of a compact smooth manifold $M$ of dimension $k$, and let $\mu$ be an ergodic $g$-invariant probability measure. Assume that $\mu$ has $l$ negative Lyapounov exponents and $k-l$ positive Lyapounov exponents for some $1\leq l\leq k-1$. There exist an irreducible countable Markov shift $(\Sigma,\sigma)$ with a $\sigma$-invariant ergodic probability measure $\nu$ and a H\"older continuous map $\pi:\Sigma\to M$ such that
\begin{equation}\label{eq:injective-symbolic-coding}
    \pi\circ\sigma=g\circ\pi,\qquad \pi\text{ is injective},\qquad \pi_*(\nu)=\mu.
\end{equation}
Moreover, there is a $Dg$-invariant splitting $T_xM=E^s(x)\oplus E^u(x)$ on $\pi(\Sigma)$ such that
\[
    \dim E^s(x)=l,\qquad \dim E^u(x)=k-l \qquad\text{for every }x\in\pi(\Sigma),
\]
with forward exponential contraction on $E^s(x)$ and backward exponential contraction on $E^u(x)$. Every periodic point of $\Sigma$ maps to a hyperbolic saddle periodic point of $g$ with stable dimension $l$.
\end{proposition}

\begin{proof}
Choose $\chi>0$ smaller than the absolute value of every Lyapunov exponent of $\mu$. By \cite[Theorem~0.1]{Ovadia18JMD}, there exist a locally compact countable Markov shift $(\Sigma_0,\sigma_0)$ and a H\"older continuous map $\pi_0:\Sigma_0\to M$ such that
\[
    \pi_0\circ\sigma_0=g\circ\pi_0, \qquad \mu\bigl(\pi_0(\Sigma_0^\#)\bigr)=1,
\]
and $\pi_0|_{\Sigma_0^\#}$ is finite-to-one. Recall that $\Sigma_0^\#$ is defined in \eqref{eq:regular}.

The restriction $\pi_0|_{\Sigma_0^\#}$ is Borel, and $g$ is a Borel automorphism of $M$. Moreover, Ben Ovadia's construction in \cite{Ovadia18JMD} shows that this restriction has the Bowen property with respect to a locally finite relation, which is also noted in \cite[Section~1.2, ``Beyond surface diffeomorphisms'']{Buzzi20JMD}. Thus Buzzi's Main Theorem \cite{Buzzi20JMD} provides a countable Markov shift\footnote{It may no longer be locally compact. See \cite[Section~1.2 and Appendix~B]{Buzzi20JMD}.} $(\widetilde\Sigma,\widetilde\sigma)$ and a $1$-Lipschitz map $\theta:\widetilde\Sigma\to\Sigma_0^\#$ such that
\[
    \widetilde\pi:=\pi_0\circ\theta: \widetilde\Sigma\longrightarrow M
\]
is injective with $g\circ \widetilde{\pi}=\widetilde{\pi}\circ \widetilde{\sigma}$ and $\mu(\widetilde\pi(\widetilde\Sigma))=1$. The map $\widetilde\pi$ is also H\"older continuous.

By the Lusin--Souslin theorem, the image of $\widetilde\pi$ is Borel and its inverse is also a Borel map from $ \widetilde{\pi}(\widetilde{\Sigma})$ to $\widetilde{\Sigma} $. Consequently, $\mu$ lifts to an ergodic $\widetilde\sigma$-invariant probability measure $\widetilde\nu$ on $\widetilde\Sigma$ with $\widetilde{\pi}_*(\widetilde{\nu})=\mu$. By \cref{lem:ergodic-component}, there is an irreducible Markov component $\Sigma\subset\widetilde\Sigma$ such that $\supp(\widetilde\nu)\subset \Sigma$. Restrict $\widetilde\sigma$, $\widetilde\pi$, and $\widetilde\nu$ to this component and denote the restrictions by $\sigma$, $\pi$, and $\nu$. Then $\pi_*\nu=\mu$, and \eqref{eq:injective-symbolic-coding} is preserved.

By \cite[Proposition~5.1]{Ovadia18JMD}, every $x\in\pi(\Sigma)\subset\pi_0(\Sigma_0)$ admits a splitting
\[
    T_xM=E^s(x)\oplus E^u(x)
\]
satisfying
\begin{equation}\label{eq:coding-contraction}
    \begin{aligned}
        \limsup_{n\to\infty}\frac1n\log\bigl\|D_xg^n|_{E^s(x)}\bigr\|\leq-\frac{\chi}{2},\qquad \limsup_{n\to\infty}\frac1n\log\bigl\|D_xg^{-n}|_{E^u(x)}\bigr\|\leq-\frac{\chi}{2}.
    \end{aligned}
\end{equation}
The splitting is $Dg$-invariant; see \cite[Proposition~2.12(3) and the proof of Proposition~5.1]{Ovadia18JMD}. The H\"older continuity in \cite[Proposition~5.1]{Ovadia18JMD}, composed with $\theta|_\Sigma$, shows that $u\mapsto E^s(\pi(u))$ and $u\mapsto E^u(\pi(u))$ are continuous. Hence the function
\[
    j:\Sigma\longrightarrow\{0,\ldots,k\}, \qquad j(u):=\dim E^s(\pi(u)),
\]
is continuous. The invariance of the splitting and the semiconjugacy identity give $j\circ\sigma=j$. Since an irreducible Markov shift always has a dense orbit and $j$ is constant on the orbit, we conclude that $j$ is constant on $\Sigma$. The ergodicity and hyperbolicity of $\mu$ implies that for $\mu$-almost every point $x\in M$ (Lyapunov regular points), we have that $\dim E^s(x)=l$ and $\dim E^u(x)=k-l$. Since $\pi_*\nu=\mu$, it follows that $j=l$ $\nu$-almost everywhere. Therefore $j$ is always equal to $l$ on $\Sigma$, and
\[
    \dim E^s(\pi(u))=l,\qquad \dim E^u(\pi(u))=k-l \quad\text{for every }u\in\Sigma.
\]

Finally, suppose $\sigma^m u=u$ for some $m\geq 1$ and set $y=\pi(u)$. Then $g^m(y)=y$. The estimates in \eqref{eq:coding-contraction} and $Dg$-invariance of the splitting space implies that $y$ is a hyperbolic saddle periodic point with stable dimension $l$ and unstable dimension $k-l$.
\end{proof}

We start the proof of \cref{thm:main}.

\begin{proof}[Proof of \cref{thm:main}]

Apply \cref{lem:coding} to the smooth real diffeomorphism underlying $f$ and to the measure $\mu$ from \eqref{eq:entropy}. We obtain an injective H\"older semiconjugacy
\[
    \pi\colon(\Sigma,\sigma)\longrightarrow(X,f), \qquad \pi_*(\nu)=\mu
\] 
where $(\Sigma,\sigma)$ is an irreducible countable Markov shift and $\nu$ is a $\sigma$-invariant ergodic probability measure on $\Sigma$. The semiconjugacy and injectivity make $(\Sigma,\sigma,\nu)$ and $(X,f,\mu)$ measure-theoretically isomorphic. In particular,
\begin{equation}\label{eq:lift-entropy}
    h_\nu(\sigma)=h_\mu(f)=\log d.
\end{equation} 
For any $\sigma$-invariant probability measure $\tau$ on $\Sigma$, since $\pi$ is an injective semiconjugacy, we have
\[
    h_\tau(\sigma)=h_{\pi_*\tau}(f)\leq h_{\mathrm{top}}(f)=\log d.
\]
The variational principle \cref{variational}, together with \eqref{eq:lift-entropy}, now gives
\[
    h(\Sigma)=\log d.
\]
Thus $\nu$ is a probability measure with maximal entropy on $\Sigma$, and the Perron value of $\Sigma$ is exactly $d$. By \cref{thm:MME-posi recurr}, the Markov shift $(\Sigma, \sigma)$ is positive recurrent. By \cref{lem:coding}, injectivity of $\pi$, and \eqref{eq:DNV}, we have
\[
    \#\Fix(\sigma^n)\leq\#SP_n\leq \# P_n\leq d^n+o(d^n).
\]
In particular, the symbolic fixed-point sets are finite. Apply \cref{prop:symbolic}, we obtain
\[
    \eta_n=d^{-n}\sum_{u\in\Fix(\sigma^n)}\delta_u\to \nu.
\]
Define $\beta_n:=\pi_*\eta_n$. For $\varphi\in C(X,\R)$, since $\pi$ is H\"older continuous, the function $\varphi\circ\pi$ is bounded
and uniformly continuous. Therefore
\[
    \int_X\varphi\,d\beta_n=\int_\Sigma\varphi\circ\pi\,d\eta_n\longrightarrow\int_\Sigma\varphi\circ\pi\,d\nu =\int_X\varphi\,d\mu.
\]
We have proved
\begin{equation}\label{eq:beta-limit}
    \beta_n\to\mu,\qquad \text{ as } n \to \infty.
\end{equation}
Define
\[
   \gamma_n:=d^{-n}\sum_{x\in SP_n}\delta_x, \qquad \zeta_n:=d^{-n}\sum_{x\in P_n}\delta_x.
\]
All these measures are positive, and injectivity of $\pi$ ensures that
\begin{equation}\label{eq:sandwich}
    0\leq\beta_n\leq\gamma_n\leq\zeta_n.
\end{equation}
The upper bound \eqref{eq:DNV} gives $\limsup_n\zeta_n(X)\leq1$, together with \eqref{eq:beta-limit} and \eqref{eq:sandwich}, we immediately have that $\zeta_n\to\mu$ and $\gamma_n\to\mu$ when $n\to \infty$. Thus, we finish the proof.
\end{proof}

% --- Bibliography ------------------------------------------------------

\printbibliography[heading=bibliography,title={References}]

\end{document}

%% file: macros.tex
\newcommand{\C}{\mathbb{C}}

\newcommand{\R}{\mathbb{R}}

\newcommand{\Z}{\mathbb{Z}}

\newcommand{\Pb}{\mathbb{P}}

\DeclareMathOperator{\Fix}{Fix}

\DeclareMathOperator{\supp}{supp}

\newcommand{\cG}{\mathcal{G}}

\newcommand{\cV}{\mathcal{V}}

\newcommand{\vep}{\varepsilon}

\newcommand{\newsharedtheorem}[2]{%
  \newaliascnt{#1}{theorem}%
  \newtheorem{#1}[#1]{#2}%
  \aliascntresetthe{#1}%
}

\theoremstyle{plain}
\newtheorem{theorem}{Theorem}[section]
\newsharedtheorem{lemma}{Lemma}
\newsharedtheorem{proposition}{Proposition}
\newsharedtheorem{corollary}{Corollary}

\theoremstyle{definition}
\newsharedtheorem{definition}{Definition}
\newsharedtheorem{example}{Example}

\theoremstyle{remark}
\newsharedtheorem{remark}{Remark}

\Crefname{theorem}{Theorem}{Theorems}
\Crefname{lemma}{Lemma}{Lemmas}
\Crefname{proposition}{Proposition}{Propositions}
\Crefname{corollary}{Corollary}{Corollaries}
\Crefname{definition}{Definition}{Definitions}
\Crefname{example}{Example}{Examples}
\Crefname{remark}{Remark}{Remarks}

%% file: refs.bib
@article{Gromov2003,
  author  = {Gromov, Mikha{\"i}l},
  title   = {On the entropy of holomorphic maps},
  journal = {Enseign. Math. (2)},
  volume  = {49},
  number  = {3-4},
  year    = {2003},
  pages   = {217--235}
}

@article{Yomdin1987,
  author  = {Yomdin, Yosef},
  title   = {Volume growth and entropy},
  journal = {Israel J. Math.},
  volume  = {57},
  number  = {3},
  year    = {1987},
  pages   = {285--300},
}

@book{ALP24MAMS,
 author = {Araujo, Ermerson and Lima, Yuri and Poletti, Mauricio},
 title = {Symbolic dynamics for nonuniformly hyperbolic maps with singularities in high dimension},
 fseries = {Memoirs of the American Mathematical Society},
 series = {Mem. Am. Math. Soc.},
 volume = {1511},
 year = {2024},
 publisher = {Providence, RI: American Mathematical Society (AMS)},
}

@article{FT24JGA,
 author = {Filip, Simion and Tosatti, Valentino},
 title = {Gaps in the support of canonical currents on projective {{\(K3\)}} surfaces},
 fjournal = {The Journal of Geometric Analysis},
 journal = {J. Geom. Anal.},
 volume = {34},
 number = {3},
 pages = {14},
 note = {Id/No 76},
 year = {2024},
}

@article{FT23CJM,
 author = {Filip, Simion and Tosatti, Valentino},
 title = {Canonical currents and heights for {{\(K3\)}} surfaces},
 fjournal = {Cambridge Journal of Mathematics},
 journal = {Camb. J. Math.},
 volume = {11},
 number = {3},
 pages = {699--794},
 year = {2023},
}

@article{FT21AJM,
 author = {Filip, Simion and Tosatti, Valentino},
 title = {Kummer rigidity for {{\(K3\)}} surface automorphisms via {Ricci}-flat metrics},
 fjournal = {American Journal of Mathematics},
 journal = {Am. J. Math.},
 volume = {143},
 number = {5},
 pages = {1431--1462},
 year = {2021},
}

@article{Cantat01Acta,
 author = {Cantat, Serge},
 title = {Dynamique des automorphismes des surfaces {{\(K3\)}}},
 fjournal = {Acta Mathematica},
 journal = {Acta Math.},
 volume = {187},
 number = {1},
 pages = {1--57},
 year = {2001},
}

@article{Buzzi09ETDS,
 author = {Buzzi, J{\'e}r{\^o}me},
 title = {Maximal entropy measures for piecewise affine surface homeomorphisms},
 fjournal = {Ergodic Theory and Dynamical Systems},
 journal = {Ergodic Theory Dyn. Syst.},
 volume = {29},
 number = {6},
 pages = {1723--1763},
 year = {2009},
}

@article{Katok07JMD,
 author = {Katok, Anatole},
 title = {Fifty years of entropy in dynamics: 1958 -- 2007},
 fjournal = {Journal of Modern Dynamics},
 journal = {J. Mod. Dyn.},
 volume = {1},
 number = {4},
 pages = {545--596},
 year = {2007},
}

@article{Sarig13JAMS,
 author = {Sarig, Omri M.},
 title = {Symbolic dynamics for surface diffeomorphisms with positive entropy},
 fjournal = {Journal of the American Mathematical Society},
 journal = {J. Am. Math. Soc.},
 volume = {26},
 number = {2},
 pages = {341--426},
 year = {2013},
}

@article{Briend-Duval99,
  AUTHOR = {Briend, Jean-Yves and Duval, Julien},
     TITLE = {Exposants de {L}iapounoff et distribution des points
              p\'eriodiques d'un endomorphisme de {$\mathbb {CP}^k$}},
   JOURNAL = {Acta Math.},
  FJOURNAL = {Acta Mathematica},
    VOLUME = {182},
      YEAR = {1999},
    NUMBER = {2},
     PAGES = {143--157},
}

@article{BLS93b,
 author = {Bedford, Eric and Lyubich, Mikhail and Smillie, John},
 title = {Distribution of periodic points of polynomial diffeomorphisms of \texorpdfstring{$\mathbb {C}^2$}{c2}},
 fjournal = {Inventiones Mathematicae},
 journal = {Invent. Math.},
 volume = {114},
 number = {2},
 pages = {277--288},
 year = {1993},
}

@article{Buzzi20JMD,
 author = {Buzzi, J{\'e}r{\^o}me},
 title = {The degree of {Bowen} factors and injective codings of diffeomorphisms},
 fjournal = {Journal of Modern Dynamics},
 journal = {J. Mod. Dyn.},
 volume = {16},
 pages = {1--36},
 year = {2020},
}

@article{DDG10ENS,
 author = {Diller, Jeffrey and Dujardin, Romain and Guedj, Vincent},
 title = {Dynamics of meromorphic maps with small topological degree. {III}: {Geometric} currents and ergodic theory},
 fjournal = {Annales Scientifiques de l'{\'E}cole Normale Sup{\'e}rieure. Quatri{\`e}me S{\'e}rie},
 journal = {Ann. Sci. {\'E}c. Norm. Sup{\'e}r. (4)},
 volume = {43},
 number = {2},
 pages = {235--278},
 year = {2010},
}

@article{dTH-Dinh12AM,
 author = {de Th{\'e}lin, Henry and Dinh, Tien-Cuong},
 title = {Dynamics of automorphisms on compact {K{\"a}hler} manifolds},
 fjournal = {Advances in Mathematics},
 journal = {Adv. Math.},
 volume = {229},
 number = {5},
 pages = {2640--2655},
 year = {2012},
}

@article {DNV,
    AUTHOR = {Dinh, Tien-Cuong and Nguy\^en, Vi\^et-Anh and Vu, Duc-Viet},
     TITLE = {Super-potentials, densities of currents and number of periodic
              points for holomorphic maps},
   JOURNAL = {Adv. Math.},
  FJOURNAL = {Advances in Mathematics},
    VOLUME = {331},
      YEAR = {2018},
     PAGES = {874--907},
}

@article {DS03JMPA,
    AUTHOR = {Dinh, Tien-Cuong and Sibony, Nessim},
     TITLE = {Dynamique des applications d'allure polynomiale},
   JOURNAL = {J. Math. Pures Appl. (9)},
  FJOURNAL = {Journal de Math\'ematiques Pures et Appliqu\'ees. Neuvi\`eme
              S\'erie},
    VOLUME = {82},
      YEAR = {2003},
    NUMBER = {4},
     PAGES = {367--423},
}

@article {DS05JAMS,
    AUTHOR = {Dinh, Tien-Cuong and Sibony, Nessim},
     TITLE = {Green currents for holomorphic automorphisms of compact
              {K}\"ahler manifolds},
   JOURNAL = {J. Amer. Math. Soc.},
  FJOURNAL = {Journal of the American Mathematical Society},
    VOLUME = {18},
      YEAR = {2005},
    NUMBER = {2},
     PAGES = {291--312},
}

@article{OT15JMST,
 author = {Oguiso, Keiji and Truong, Tuyen Trung},
 title = {Explicit examples of rational and {Calabi}-{Yau} threefolds with primitive automorphisms of positive entropy},
 fjournal = {Journal of Mathematical Sciences. University of Tokyo},
 journal = {J. Math. Sci., Tokyo},
 volume = {22},
 number = {1},
 pages = {361--385},
 year = {2015},
}

@article {DS10JAG,
    AUTHOR = {Dinh, Tien-Cuong and Sibony, Nessim},
     TITLE = {Super-potentials for currents on compact {K}\"ahler manifolds
              and dynamics of automorphisms},
   JOURNAL = {J. Algebraic Geom.},
  FJOURNAL = {Journal of Algebraic Geometry},
    VOLUME = {19},
      YEAR = {2010},
    NUMBER = {3},
     PAGES = {473--529},
}

@article{DS16,
  author  = {Dinh, Tien-Cuong and Sibony, Nessim},
  title   = {Equidistribution of saddle periodic points for 
             {H\'enon}-type automorphisms of 
             {$\mathbb C^k$}},
  fjournal = {Mathematische Annalen},
  journal  = {Math. Ann.},
  volume   = {366},
  number   = {3-4},
  pages    = {1207--1251},
  year     = {2016},
}

@article{DS17,
 author = {Dinh, Tien-Cuong and Sibony, Nessim},
 title = {Equidistribution problems in complex dynamics of higher dimension},
 fjournal = {International Journal of Mathematics},
 journal = {Int. J. Math.},
 volume = {28},
 number = {7},
 eid = {1750057},
 year = {2017},
}

@article{Gurevich69,
 author = {Gurevich, B. M.},
 title = {Topological entropy of enumerable {Markov} chains},
 fjournal = {Soviet Mathematics. Doklady},
 journal = {Sov. Math., Dokl.},
 volume = {10},
 pages = {911--915},
 year = {1969},
}

@article{Gurevich70,
 author = {Gurevich, B. M.},
 title = {Shift entropy and {Markov} measures in the path space of a denumerable graph},
 fjournal = {Soviet Mathematics. Doklady},
 journal = {Sov. Math., Dokl.},
 volume = {11},
 pages = {744--747},
 year = {1970},
}

@book{Kitchens,
 author = {Kitchens, Bruce P.},
 title = {Symbolic dynamics. {One}-sided, two-sided and countable state {Markov} shifts},
 fseries = {Universitext},
 series = {Universitext},
 year = {1998},
 publisher = {Berlin: Springer},
}

@article{LZ25,
  author  = {Luo, Muhan and Zhou, Qi},
  title   = {Equidistribution of saddle periodic points for {H}{\'e}non-like maps},
  journal = {J. Eur. Math. Soc. (JEMS)},
  note    = {to appear},
}

@article{Lyubich83FAA,
 author = {Lyubich, M. Yu.},
 title = {The maximum-entropy measure of a rational endomorphism of the {Riemann} sphere},
 fjournal = {Functional Analysis and its Applications},
 journal = {Funct. Anal. Appl.},
 volume = {16},
 pages = {309--311},
 year = {1983},
}

@article{Ovadia18JMD,
 author = {Ben Ovadia, Snir},
 title = {Symbolic dynamics for non-uniformly hyperbolic diffeomorphisms of compact smooth manifolds},
 fjournal = {Journal of Modern Dynamics},
 journal = {J. Mod. Dyn.},
 volume = {13},
 pages = {43--113},
 year = {2018},
}

@book{Durrett99ESP,
 author = {Durrett, Rick},
 title = {Essentials of stochastic processes},
 fseries = {Springer Texts in Statistics},
 series = {Springer Texts Stat.},
 year = {1999},
 publisher = {New York, NY: Springer},
}

@article {DurjaDuke,
    AUTHOR = {Dujardin, Romain},
     TITLE = {Laminar currents and birational dynamics},
   JOURNAL = {Duke Math. J.},
  FJOURNAL = {Duke Mathematical Journal},
    VOLUME = {131},
      YEAR = {2006},
    NUMBER = {2},
     PAGES = {219--247},
}

@article {deT08Invent,
    AUTHOR = {de Th\'elin, Henry},
     TITLE = {Sur les exposants de {L}yapounov des applications
              m\'eromorphes},
   JOURNAL = {Invent. Math.},
  FJOURNAL = {Inventiones Mathematicae},
    VOLUME = {172},
      YEAR = {2008},
    NUMBER = {1},
     PAGES = {89--116},
}
